\long\def\skipit#1{} 

\newcommand{\llceil}{\left\lceil}
\newcommand{\llfloor}{\left\lfloor}
\newcommand{\lra}{\longrightarrow}

\newcommand{\mdef}[1]{\textit{\textbf{#1}}}
\newcommand{\noi}{\noindent}

\newcommand{\remark}{\noi\textsc{Remark}. }
\newcommand{\rrceil}{\right\rceil}
\newcommand{\rrfloor}{\right\rfloor}

\newcommand{\vsa}{\vskip-12pt}
\newcommand{\vsb}{\vskip-6pt}

\documentclass[12pt]{amsart}
\advance\textheight by \topskip
\usepackage{amsfonts,amsmath,amssymb}
\usepackage{epsfig,graphicx}
\usepackage{latexsym,ifthen,calc}
\usepackage{multicol}
\usepackage[format=hang, margin=10pt]{caption}

\usepackage{xcolor}

\newcounter{hours}
\newcounter{minutes}
\newcommand{\printtime}{
	\setcounter{hours}{\time/60}%
	\setcounter{minutes}{\time-\value{hours}*60}
	\ifthenelse{\value{hours}<10}{0}{}\thehours:%
	\ifthenelse{\value{minutes}<10}{0}{}\theminutes}

\numberwithin{equation}{section}
\numberwithin{figure}{section}
\numberwithin{table}{section}

\newtheorem{thm}{Theorem}[section]
\newtheorem{cor}[thm]{Corollary}
\newtheorem{lem}[thm]{Lemma}
\newtheorem{prop}[thm]{Proposition}
\newtheorem{J-com}{JG-comment}[section]
\theoremstyle{definition}

\begin{document}

\title{{Iterated Claws Have Real-Rooted Genus Polynomials}}

\author[J.L. Gross, T. Mansour, T.W. Tucker, and D.G.L. Wang]{Jonathan L. Gross
}
\address{
Department of Computer Science  \\
Columbia University, New York, NY 10027, USA; \newline
email: gross@cs.columbia.edu
}
\author[]{Toufik Mansour
}
\address{
Department of Mathematics  \\
University of Haifa, 3498838 Haifa, Israel;  \newline
email: tmansour@univ.haifa.ac.il}
\author[]{Thomas W. Tucker
}
\address{
Department of Mathematics  \\
Colgate University, Hamilton, NY 13346, USA; \newline
email: ttucker@colgate.edu
}
\author[]{David G.L. Wang
}
\address{
School of Mathematics and Statistics  \\
Beijing Institute of Technology, 102488 Beijing, P. R. China;  \newline
email: glw@bit.edu.cn}

\date{}

\begin{abstract}
We prove that the genus polynomials of the graphs called \textit{iterated claws} are real-rooted.  This continues our work directed toward the 25-year-old conjecture that the genus distribution of every graph is log-concave.  We have previously established log-concavity for sequences of graphs constructed by iterative vertex-amalgamation or iterative edge-amalgamation of graphs that satisfy a commonly observable condition on their partitioned genus distributions, even though it had been proved previously that iterative amalgamation does not always preserve real-rootedness of the genus polynomial of the iterated graph.  In this paper, the iterated topological operations are adding a claw and adding a 3-cycle, rather than vertex- or edge-amalgamation.  Our analysis here illustrates some advantages of employing a matrix representation of the transposition of a set of productions.
\end{abstract}


\maketitle
%

\enlargethispage{-12pt}

\section{\large{Introduction}}

Graphs are implicitly taken to be connected.  Our \mdef{graph embeddings} are cellular and orientable.  For general background in topological graph theory, see \cite{BW09B,GrTu87}.  Prior acquaintance with the concepts of \textit{partitioned genus distribution} (abbreviated here as \mdef{pgd}) and \textit{production} (e.g., \cite{GKP10, PKG10}) is prerequisite to reading this paper.  Subject to this prerequisite, the exposition here is intended to be accessible both to graph theorists and to combinatorialists.

The \mdef{genus distribution} of a graph $G$ is the sequence
$g_0(G)$, $g_1(G)$, $g_2(G)$, $\ldots$,
where $g_i(G)$ is the number of combinatorially distinct embeddings of~$G$
in the orientable surface of genus~$i$.  A genus distribution contains only finitely many positive numbers, and  there are no zeros between the first and last positive numbers.  The \mdef{genus polynomial} is the polynomial
$$\Gamma_G(z) \,=\, g_0(G) + g_1(G)z +g_2(G)z^2 +\ldots$$

We say that a sequence $A=(a_k)_{k=0}^n$ is \mdef{nonnegative} if $a_k\ge0$ for all~$k$.  
An element~$a_k$ is said to be an \mdef{internal zero} of~$A$ if there exist indices $i$ and $j$ with  $i<k<j$, such that $a_ia_j\ne 0$ and $a_k=0$. If $a_{k-1}a_{k+1} \le a_k^2$ for all $k$, then $A$ is said to be \mdef{log-concave}.  If there exists an index~$h$ with $0\le h\le n$ such that
\[
a_0\,\le\, a_1\,\le\, \cdots\,\le\, a_{h-1}\,\le\, a_h\,\ge\, a_{h+1}\,\ge\,\cdots\,\ge\, a_n,
\]
then $A$ is said to be \mdef{unimodal}.  It is well-known that any nonnegative log-concave sequence without internal zeros is unimodal, and that any nonnegative unimodal sequence has no internal zeros. A prior paper \cite{GMTW13} by the present authors provides additional contextual information regarding log-concavity and genus distributions.

For convenience, we sometimes abbreviate the phrase ``log-concave genus distribution'' as \mdef{LCGD}.  Proofs that closed-end ladders and doubled paths have LCGDs \cite{FGS89} were based on closed formulas for their genus distributions.  Proof that bouquets have LCGDs \cite{GrRoTu89} was based on a recursion.

Stahl \cite{Stah97} used the term ``$H$-linear'' to describe chains of graphs obtained by amalgamating copies of a fixed graph $H$.  He conjectured that a number of ``$H$-linear'' families of graphs have genus polynomials with nonpositive real roots, which implies the log-concavity of their sequences of coefficients, by Newton's theorem.  Although it was shown \cite{Wag97} that the genus polynomials of some such families do indeed have real roots, Stahl's conjecture of real-rootedness for $W_4$-linear graphs (where $W_4$ is the 4-wheel) was disproved by Liu and Wang \cite{LW07}.  

Our previous paper \cite{GMTW13} proves, nonetheless, that the genus distribution of every graph in the $W_4$-linear sequence is log-concave.  Thus, even though Stahl's proposed approach to log-concavity via roots of genus polynomials is sometimes infeasible, \cite{GMTW13} does support Stahl's expectation that chains of copies of a graph are a relatively accessible aspect of the general LCGD problem.  Moreover, Wagner \cite{Wag97} has proved the real-rootedness of the genus polynomials for a number of graph families for which Stahl made specific conjectures of real-rootedness.  

Furthermore, we shall see here that Stahl's method of representing what we have elsewhere presented as a transposition of a \textit{production system} for a surgical operation on graph embeddings as a matrix of polynomials can simplify a proof, that a family of graphs has log-concave genus distributions.

\enlargethispage{12pt}

\medskip
\section{\large{The Sequence of Iterated Claws}}  

Let the rooted graph $(Y_0,u_0)$ be isomorphic to the dipole $D_3$, and let the root $u_0$ be either vertex of $D_3$.   For $n=1, 2, \ldots$, we define the \mdef{iterated claw} $(Y_n,u_n)$ to be the graph obtained the following surgical operation:
\begin{quote}
\mdef{Newclaw}:  Subdivide each of the three edges incident on the root vertex $u_{n-1}$ of the iterated claw $(Y_{n-1},u_{n-1})$, and then join the three new vertices obtained thereby to a new root vertex $u_n$.
\end{quote}
Figure \ref{fig:Y3} illustrates the graph $(Y_3,u_3)$.
\enlargethispage{-12pt}

\begin{figure} [ht]
\centering 
\includegraphics[width=0.9\textwidth,natwidth=380,natheight=110]{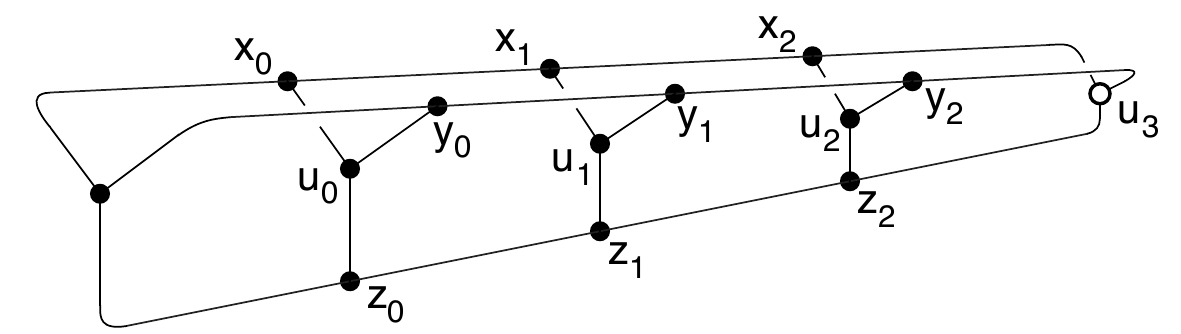}  
\caption{The rooted graph $(Y_3,u_3)$.}
\label{fig:Y3}
\vskip-3pt
\end{figure}  
\medskip

\noi The graph $K_{1,3}$ is commonly called a \mdef{claw graph}, which accounts for our name \textit{iterated claw}.  The notation $Y_n$ reflects the fact that a claw graph looks like the letter $Y$.  We observe that $Y_1\cong K_{3,3}$.   A recursion for the genus distribution of the iterated claw graphs is derived in \cite{GKP12}.  {We observe that, whereas all of Stahl's examples \cite{Stah97} of graphs with log-concave genus distributions are planar, the sequence of iterated claws has rising minimum genus.}

We have seen in previous studies of genus distribution (especially \cite{Gr12}) that the number of productions and simultaneous recursions rises rapidly with the number of roots and the valences of the roots.  The surgical operation newclaw is designed to circumvent this problem.

For a single-rooted iterated claw $(Y_n,u_n)$, we can define three \mdef{partial genus distributions}, also called \mdef{partials}.  Let
\begin{eqnarray*}
a_{n,i} &=& \text{the number of embeddings $Y_n\to S_i$ such that} \\[-3pt]
	&& \text{three different fb-walks are incident on the root $u_n$;} \\
b_{n,i} &=& \text{the number of embeddings $Y_n\to S_i$ such that exactly} \\[-3pt]
	&& \text{two different fb-walks are incident on the root $u_n$;}\\
c_{n,i} &=& \text{the number of embeddings $Y_n\to S_i$ such that} \\[-3pt]
	&& \text{one fb-walk is incident three times on the root $u_n$.}  
\end{eqnarray*}
We also define the generating functions 
\begin{align*}
\noalign{\vsb}
A_n(z) &\,=\, \sum_{i=0}^\infty a_{n,i}z^i \\
B_n(z) &\,= \,\sum_{i=0}^\infty b_{n,i}z^i \\
C_n(z) &\,=\,\sum_{i=0}^\infty c_{n,i}z^i  
\end{align*} 

A listing of the non-zero values of all the partials for every genus $i$ is called a \mdef{partitioned genus distribution}.  Clearly, the full genus distribution is the sum of the partials.  That is, for $i\,=\,0,1,2, \ldots$, we have
\begin{eqnarray*}
\noalign{\vskip-6pt}
g_i(Y_n) &=& a_{n,i}\,+\,b_{n,i}\,+\,c_{n,i} \\[-6pt]
\noalign{\noi and}
\Gamma_{Y_n}(z) &=& A_n(z)\,+\,B_n(z)\,+\,C_n(z) 
\end{eqnarray*}
We define $g_{n,i}=g_i(Y_n)$. 
\smallskip

\begin{thm}
For $n>1$, the effect on the pgd of applying the operation newclaw to the iterated claw $(Y_{n-1},u_{n-1})$ corresponds to the following system of three productions:
\begin{alignat}{7}
&a_i  \;\lra\;   &&\phantom{\;+\;} 12b_{i+1} &&\;+ 4c_{i+2}\label{prod:a*}\\
&b_i  \;\lra\;  ~~2a_i &&\;+ 12b_{i+1} &&\;+ 2c_{i+1} \label{prod:b*}\\
&c_i  \;\lra\;  8a_i && &&\;+  8c_{i+1} \label{prod:c*}
\end{alignat} 
\end{thm}

\begin{proof} \vsb
This is Theorem 4.5 of \cite{GKP12}.
\end{proof}
\smallskip

\begin{cor} \label{cor:pgd-claws}
For $n>1$, the effect on the pgd of applying the operation newclaw to the iterated claw $(Y_{n-1},u_{n-1})$ corresponds to the following recurrence relations:
\begin{alignat}{7}
&a_{n,i} \;=\; &&\phantom{\;+\;}\hskip10pt  2b_{n-1,i}   &&\;+ ~~8c_{n-1,i} \label{rec:a}\\
&b_{n,i} \;=\; 12a_{n-1,i-1} &&\;+\; 12b_{n-1,i-1} \;&& \label{rec:b}\\
&c_{n,i} \;=\;  \hskip6pt 4a_{n-1,i-2} &&\;+\;  \hskip5pt 2b_{n-1,i-1}  &&\;+ ~~8c_{n-1,i-1} \label{rec:c} 
\end{alignat} 
\end{cor}

\begin{proof}\vsb
The recurrence system \eqref{rec:a}, \eqref{rec:b}, \eqref{rec:c} is reasonably described as a \mdef{transposition of the production system} \eqref{prod:a*}, \eqref{prod:b*}, \eqref{prod:c*}.
\end{proof}

It is convenient to express such a recurrence system in matrix form:
\begin{equation}
V(Y_n) = M(z)\cdot V(Y_{n-1})
\end{equation}
with the initial column vector
\begin{equation}
V(Y_0) ~= ~
\left[\begin{matrix} A(Y_0) \\ B(Y_0) \\ C(Y_0) \end{matrix}\right]
~=~  \left[\begin{matrix} 2 \\ 0 \\ 2z \end{matrix}\right] 
\end{equation}
and the \mdef{production matrix}
\begin{equation}
M(z)=
\left[\begin{matrix}
0 & 2 & 8 \\
12z & 12z &0 \\
~4z^2&~ 2z & 8z
\end{matrix}\right]
\end{equation}
\smallskip

\begin{prop}
The column vector $V(Y_n)$ is the product of the matrix power $M^{n}(z)$ with the column vector $V(Y_0)$.
\end{prop}
\smallskip

\begin{cor}\label{cor:co1}
The column vector $V(Y_n)$ is the product of the matrix power $M^{n+1}(z)$ with the (artificially labeled) column vector
$$V(Y_{-1}) = \begin{pmatrix} 0 \\ 0 \\ 1/4 \end{pmatrix}$$
\end{cor}
\smallskip

\begin{cor}  \label{cor:column3}
To prove that every iterated claw has an LCGD, it is sufficient to prove that the sum of the third column of the matrix $M^{n}(z)$ is a log-concave polynomial.
\end{cor}
\smallskip

\remark  Partitioned genus distributions and recursion systems for pgds were first used by Furst, Gross, and Statman \cite{FGS89}.  Stahl \cite{Stah97} was first to employ a matrix equivalent of a production system to investigate log-concavity.   A presentation by Mohar \cite{Mo12} has called more recent attention to the matrix equivalent and its properties.

\bigskip
\section{\large{Characterizing Genus Polynomials for Iterated Claws}\label{sec:char}}  

In this section, we investigate some properties of the genus polynomials of iterated claws.  Corollary \ref{cor:column3} leads us to focus on the sum of the third column of the matrix $M^{n}(z)$, which is expressible as $(1,1,1)M^n(z)(4V(Y_{-1}))$, which implies that it equals {\hbox{4 times} the genus polynomial} of the iterated claw $Y_{n-1}$.  Theorem \ref{th1} formulates a generating function $f(z,t)$ for this sequence of sums, and Theorem \ref{th2} uses the generating function to construct an expression for the genus polynomials from which we establish interlacing of roots in Section \ref{sec:rr}.
\smallskip

\begin{thm}\label{th1}
The generating function $f(z,t) = \sum_{n\geq0} (1,1,1)M^n(z)(4V(Y_{-1}))t^n$ for the sequence of sums of the third column of $M^n(z)$ has the closed form
\begin{equation} \label{gf1}
f(z,t) ~=~ \frac{1+(8-12z)t-24zt^2}{1-20zt+8z(8z-3)t^2+384z^3t^3}.
\end{equation}
\end{thm}

\begin{proof}
Let $(p_n,q_n,rq_n) = (1,1,1)M^n$ for all $n\geq0$. Then
\begin{align}
(p_{n+1},q_{n+1},r_{n+1})&=(p_n,q_n,r_n)M(z) \label{eq1}\\
&=(12zq_n+4z^2r_n,\ 2p_n+12zq_n+2zr_n,\ 8p_n+8zr_n). \notag
\end{align}
The third coordinate of Equation \eqref{eq1} implies that
\begin{align}
p_n \,=\, \frac{1}{8}(r_{n+1}-8zr_n).\label{eq2}
\end{align}
By combining \eqref{eq2} with the first coordinate of (\ref{eq1}) we obtain
\begin{align}
q_n \,=\, \frac{1}{96z}(r_{n+2}-8zr_{n+1}-32z^2r_n).\label{eq3}
\end{align}
The second coordinate of (\ref{eq1}) yields
\begin{equation} \label{eq:b}
q_{n+1} \,=\, 2p_n+12zq_n+2zr_n
\end{equation}
Substituting (\ref{eq2}) and (\ref{eq3}) into \eqref{eq:b} leads to the recurrence relation
\begin{equation}\label{rec:C}
r_{n} \,=\, 20zr_{n-1}+8z(3-8z)r_{n-2}-384z^3r_{n-3}
\end{equation}
with
\begin{equation}\label{ini:C}
\begin{aligned}
r_0& \,=\, 1,\\
r_1& \,=\, 8+8z,\\
r_2& \,=\, 160z+96z^2.
\end{aligned}
\end{equation}
By multiplying Recurrence \eqref{rec:C} by $t^n$ and summing over all $n\geq0$, we obtain Generating Function \eqref{gf1}.
\end{proof}
\smallskip

\begin{thm}\label{th2}
The genus polynomial of the iterated claw $Y_n$ is given by
\begin{align*}
\noalign{\vskip-3pt}
(1,1,1)M^{n+1}(z)V(Y_{-1}) ~=~ 2^{n-1}(h_{n+1}(z)+2(2-3z)h_{n}(z)-6zh_{n-1}(z)), \\[-24pt]
\end{align*}
where
\begin{align*}
h_{n}(z)&\quad=\sum_{2j+i_1+i_2+i_3=n}\binom{j+i_1}{i_1}\binom{j+i_2}{i_2}\binom{j+i_3}{i_3}(1+\sqrt{3})^{i_2}(1-\sqrt{3})^{i_3}3^{j+i_1}(2z)^{n-j}.
\end{align*}
\end{thm}

\begin{proof}
By Theorem \ref{th1}, we have
$$f(z,t) ~=~ \sum_{n\geq0}(1,1,1)M^n(4V(Y_0))t^n ~=~ \frac{1+(8-12z)t-24zt^2}{1-20zt+8z(8z-3)t^2+384z^3t^3}.$$
Thus,
\begin{align*}
f(z/2,t/2)&~=~ \frac{1+(4-3z)t-3zt^2}{1-5zt+z(4z-3)t^2+6z^3t^3}\\
&~=~ \frac{1+(4-3z)t-3zt^2}{(1-2zt-2z^2t^2)(1-3zt)-3zt^2}\\
&~=~ \sum_{j\geq0}\frac{(1+(4-3z)t-3zt^2)3^jz^jt^{2j}}{(1-3zt)^{j+1}(1+\sqrt{3}zt)^{j+1}(1-\sqrt{3}zt)^{j+1}}.
\end{align*}
Using the combinatorial identity $(1-at)^{-m}=\sum_{j\geq0}\binom{m-1+j}{j}a^jt^j$, and then finding the coefficient of $t^n$, we derive the equation
\begin{align*}
(1,1,1)M^n(z/2)V(Y_0) ~=~ 2^{n-2}(h_{n}(z)+2(2-3z)h_{n-1}(z)-6zh_{n-2}(z)),
\end{align*}
which, by Corollary \ref{cor:co1}, completes the proof.
\end{proof}
\smallskip

Theorem \ref{th2} provides an explicit expression for the genus polynomial $\Gamma_{Y_n}(z)$.
It is easy to see that $\Gamma_{Y_n}(z)=r_{n+1}/4$,
where $r_n$ is defined in the proof of Theorem~\ref{th1}.
In terms of $\Gamma_{Y_n}(z)$,
the recurrence relation~\eqref{rec:C} becomes
\begin{equation}\label{rec:Gamma}
\Gamma_{Y_n}(z) \,=\, 20z\Gamma_{Y_{n-1}}(z)+8z(3-8z)\Gamma_{Y_{n-2}}(z)-384z^3\Gamma_{Y_{n-3}}(z).
\end{equation}

Let $g_{n,i}$ be the coefficient of $z^i$ in $\Gamma_{Y_n}(z)$.
The following table of values of $g_{n,i}$ for $n\le4$ is derived in \cite{GKP12}.
\begin{center}
\begin{tabular}{|c|cccccc|}
\hline
$g_{n,i}$ & $i=0${\vphantom{$2^{M^M}$}} & 1 &\ ~2~\ &\ ~3~\ &\ ~4~\ &\ ~5~\ \\[2pt]
\hline
$n=0$ {\vphantom{$2^{M^M}$}}& 2 & 2 & 0  & 0 & 0 &0 \\
1 & 0 & 40 & 24 & 0 & 0 & 0 \\
2 & 0 & 48 & 720 & 256 & 0 & 0 \\
3 & 0 & 0  & 1920 & 11648 & 2816 & 0 \\
4 & 0 & 0 & 1152 & 52608  &  177664  & 30720 \\
\hline
\end{tabular}
\end{center}

Denote by $\mathcal{P}_{s,t}$ the set of polynomials of the form $\sum_{k=s}^ta_kz^k$,
where $a_k$ is a positive integer for any $s\le k\le t$.
The above table suggests that $\Gamma_{Y_n}(z)\in\mathcal{P}_{\lfloor (n+1)/2\rfloor,\,n+1}$
for $n\le4$. Now we show it holds true in general.

\smallskip

\begin{thm}\label{thm:C}
For all $n\ge0$, the polynomial $\Gamma_{Y_n}(z)\in\mathcal{P}_{\lfloor (n+1)/2\rfloor,\,n+1}$.
Moreover, we have the leading coefficient
\begin{equation}\label{exp:n:n+1}
g_{n,n+1} \,=\, 4^{n}\sum_{k=0}^{\lfloor(n+1)/2\rfloor}\binom{n+2}{2k+1}3^k,
\end{equation}
and, for any number $i$ such that $\lfloor (n+1)/2\rfloor+1 \,\le\, i\le n$, we have
\begin{equation}\label{ineq:11}
g_{n,i} \,>\, 11g_{n-1,i-1}.
\end{equation}
\end{thm}

\begin{proof}
We see in the table above that $\Gamma_{Y_n}(z)\in\mathcal{P}_{\lfloor (n+1)/2\rfloor,\,n+1}$ and that Equation~\eqref{exp:n:n+1} and Inequality~\eqref{ineq:11} are true, for $n\le 4$.  Now suppose that $n\ge5$.
For convenience, let $g_{k,i}=0$ for all $i<0$.
We can also take $g_{k,i}=0$ for $i>k+1$, by induction using~\eqref{rec:Gamma}, for $k<n$.
From Recurrence~(\ref{rec:Gamma}) and the induction hypothesis, we have
\begin{equation}\label{rec:i}
g_{n,i} \,=\, 20g_{n-1,i-1}+24g_{n-2,i-1}-64g_{n-2,i-2}-384g_{n-3,i-3},
\qquad n\ge3.
\end{equation}

For $i>n+1$, the induction hypothesis implies that each of the four terms on the right side of
Recurrence \eqref{rec:i} is zero-valued.  So the degree of $\Gamma_{Y_n}(z)$ is at most $n+1$.
Let $s_i \,=\, g_{i,i+1}$. Taking $i=n+1$ in~(\ref{rec:i}), we get
\begin{equation}\label{rec:s}
s_n \,=\, 20s_{n-1}-64s_{n-2}-384s_{n-3},•
\end{equation} 
with the initial values $s_0=2$, $s_1=24$, $s_2=256$.
The above recurrence can be solved by a standard generating function method,
see~\cite[p.8]{Wilf06}.
In practice, 
we use the command {\tt rsolve} in the software Maple 
and get the explicit formula directly as 
\[
s_n \,=\, 4^{n}\sum_{k\ge0}\binom{n+2}{2k+1}3^k.
\]
It follows that $g_{n,n+1}>0$. Hence the degree of $\Gamma_{Y_n}(z)$ is exactly $n+1$.

\enlargethispage{12pt}

Similarly, for $i<\lfloor (n+1)/2\rfloor$, the four terms on the right side of \eqref{rec:i} are zero-valued, so the minimum genus of $Y_n$ is at least $\lfloor (n+1)/2\rfloor$.  Moreover, applying \eqref{rec:i} with $i=\lfloor (n+1)/2\rfloor$ and using the induction hypothesis $g_{k,i}=0$ for all $i<\lfloor (k+1)/2\rfloor$ with $k<n$, we find the first term is positive for $n$ odd and zero for $n$ even,
the second term is always positive, and the third and fourth terms are always zero.
In other words,
\[
g_{n,\lfloor (n+1)/2\rfloor} \,=\, 20g_{n-1,\lfloor (n+1)/2\rfloor-1}+24g_{n-2,\lfloor (n+1)/2\rfloor-1}\ge24g_{n-2,\lfloor (n+1)/2\rfloor-1}>0.
\]
This confirms the minimum genus of $Y_n$ is exactly $\lfloor (n+1)/2\rfloor$.

Now consider $i$ such that $\lfloor (n+1)/2\rfloor+1 \,\le\, i\le n$.  By~(\ref{rec:i}), and using~\eqref{ineq:11} inductively, we deduce
\begin{align*}
g_{n,i}\,
& =11g_{n-1,i-1}+24g_{n-2,i-1}+(9g_{n-1,i-1}-64g_{n-2,i-2}-384g_{n-3,i-3})\\
&>11g_{n-1,i-1}+24g_{n-2,i-1}+(35g_{n-2,i-2}-384g_{n-3,i-3})\\
&>11g_{n-1,i-1}+24g_{n-2,i-1}+g_{n-3,i-3}\\
&\ge11g_{n-1,i-1}.
\end{align*}
So Inequality~(\ref{ineq:11}) holds true.
It follows that $g_{n,i}>0$. Hence $$\Gamma_{Y_n}(z)\in\mathcal{P}_{\lfloor (n+1)/2\rfloor,n+1}.$$
This completes the proof.
\end{proof}
\smallskip

\begin{cor} \label{cor:interp}
The sequence of coefficients of $\Gamma_{Y_n}(z)$ has no internal zeros.
\end{cor}

\begin{proof}\vsb
This special case of the familiar and easily proved ``interpolation theorem'' of topological graph theory also follows directly from Theorem \ref{thm:C}.
\end{proof}

\bigskip
\section{\large{Genus Polynomials for Iterated Claws are Real-Rooted}\label{sec:rr}}  

Our goal in this section is to establish in Theorem \ref{thm:RR} the real-rootedness of the genus polynomials $\Gamma_{Y_n}(z)$ of the iterated claws, via an associated sequence $W_n(z)$ of normalized polynomials.  It follows from this real-rootedness, by Newton's theorem (e.g., see \cite{Stan89}, Theorem 2), that the genus polynomials for iterated claws are log-concave.

To proceed, we ``normalize'' the polynomials $\Gamma_{Y_n}(z)$ by defining
\begin{equation}\label{def:W}
W_n(z) ~=~ z^{-\lfloor (n+1)/2\rfloor}\Gamma_{Y_n}(z),
\end{equation}
so that $W_n(z)$  starts from a non-zero constant term,
and has the same non-zero roots as $\Gamma_{Y_n}(z)$.
We use the symbol $d_n$ to denote the degree of $W_n(z)$, that is,
\begin{equation}
 d_n \,=\, \deg{W_n(z)} \,=\, (n+1)-\llfloor \frac{n+1}{2}\rrfloor \,=\, \llceil \frac{n+1}{2}\rrceil.
\end{equation}
By Theorem~\ref{thm:C}, we have
$W_n(z)\in\mathcal{P}_{0,d_n}$.
Substituting~(\ref{def:W}) into the recurrence relation~(\ref{rec:Gamma}),
we derive
\begin{equation}\label{rec:W}
W_n(z) ~=~ \begin{cases}
20zW_{n-1}(z)+8(3-8z)W_{n-2}(z)-384z^2W_{n-3}(z),&\text{if $n$ is even},\\
20W_{n-1}(z)+8(3-8z)W_{n-2}(z)-384zW_{n-3}(z),&\text{if $n$ is odd},
\end{cases}
\end{equation}
with the initial polynomials
\begin{equation}\label{ini:W}
\begin{aligned}
W_0(z)&=2(1+z),\\
W_1(z)&=8(5+3z),\\
W_2(z)&=16(3+45z+16z^2).
\end{aligned}
\end{equation}

Let $\mathcal{P}$ denote the union $\cup_{n\ge0}\mathcal{P}_{0,n}
=\cup_{n\ge0}\{\sum_{k=0}^na_kz^k\,|\,a_k\in\mathbb{Z}^+\}$.
Lemma \ref{lem:interlacing} is an elementary consequence of the
intermediate value theorem.

\begin{lem}\label{lem:interlacing}
Let $P(x),Q(x)\in\mathcal{P}$.
Suppose that $P(x)$ has roots $x_1<x_2<\cdots<x_{\deg P}$,
and that $Q(x)$ has roots $y_1<y_2<\cdots<y_{\deg Q}$.
If $\deg Q-\deg P\in\{0,1\}$ and if the roots interlace so that 
$$x_1\,<\, y_1 \,<\, x_2 \,<\, y_2 \,<\, \cdots, $$
then
\begin{align}
(-1)^{i+\deg{P}}P(y_i)>0&\qquad\text{for all $1\le i\le \deg{Q}$},\label{lem:interlacing:P}\\
(-1)^{j+\deg{Q}}Q(x_j)<0&\qquad\text{for all $1\le j\le \deg{P}$}.\label{lem:interlacing:Q}
\end{align}
\end{lem}

\begin{proof}
Since $P(x)$ is a polynomial with positive coefficients, we have
\begin{equation}\label{sgn:lem:-infty}
	(-1)^{\deg{P}}P(-\infty)>0.
\end{equation}
We suppose first that $\deg{P(x)}$ is odd, and we consider the curve~$P(x)$.  We see that Inequality~(\ref{sgn:lem:-infty}) reduces to $P(-\infty)<0$.  Thus, the curve~$P(x)$ starts in the lower half plane and intersects the $x$-axis at its first root,~$x_1$.  From there, the curve~$P(x)$ proceeds without going below the $x$-axis, until it meets the second root,~$x_2$. Since $x_1<y_1<x_2$, we recognize that~(\ref{lem:interlacing:P}) holds for $i=1$, i.e., 
\begin{equation}\label{sgn:lem:y1}
	P(y_1)>0.
\end{equation}
After passing through~$x_2$, the curve $P(x)$ stays below the $x$-axis up to the third root,~$x_3$.  It is clear that the curve $P(x)$ continues going forward, intersecting the $x$-axis in this alternating way.  It follows from this alternation that 
\begin{equation}\label{sgn:lem:PP<0}
P(y_k)P(y_{k+1})<0\qquad\text{for all $1\le k\le \deg Q-1$}.
\end{equation}
From~(\ref{sgn:lem:y1}) and ~\eqref{sgn:lem:PP<0}, we conclude that~(\ref{lem:interlacing:P}) holds for all $1\le i\le \deg{Q}$, when $\deg{P(x)}$ is odd. 

We next suppose that $\deg{P(x)}$ is even.  In this case, we can draw the curve $P(x)$ so that it starts in the upper half plane, first intersects the $x$-axis at $x_1$, then goes below the axis up to $x_2$, and continues alternatingly.  Therefore the sign-alternating relation~(\ref{sgn:lem:PP<0}) still holds.  Since $P(y_1)<0$ when $\deg{P(x)}$ is even, we have proved (\ref{lem:interlacing:P}).

It is obvious that Inequality~(\ref{lem:interlacing:Q}) can be shown
along the same line. This completes the proof of Lemma \ref{lem:interlacing}.
\end{proof}
\smallskip

Now we proceed with our main theorem on the genus polynomial of iterated claws.

\begin{thm}\label{thm:RR}
For every $n\ge0$, the polynomial $W_n(z)$ is real-rooted.  Moreover, if the roots of $W_k(z)$ are denoted by $x_{k,1}<x_{k,2}<\cdots$, then we have the following interlacing properties:
\begin{itemize}
\item[(i)]
for every $n\ge2$, the polynomial $W_n(z)$ has one more root than $W_{n-2}(z)$, and the roots interlace so that \vsa\vsb
$$
\, x_{n,1} \,<\, x_{n-2,1} \,<\, x_{n,2} \,<\, x_{n-2,2} \,<\, \cdots \,<\, x_{n,d_n-1} \,<\, x_{n-2,d_n-1} \,<\, x_{n,d_n};
$$
\item[(ii)]
for every $n\ge1$, the polynomial $W_n(z)$ has either one more (when $n$ is even) or the same number (when $n$ is odd) of roots as $W_{n-1}(z)$, and the roots interlace so that \vsa\vsb
$$\qquad \ \ \ x_{n,1} \,<\, x_{n-1,1} \,<\, x_{n,2} \,<\, x_{n-1,2} \,<\, \cdots \,<\, x_{n-1,d_n-1} \,<\, x_{n,d_n}\quad\text{when $n$ even};$$ \vskip-3pt 
\noi and \vsa\vsb
$$\qquad\qquad x_{n,1} \,<\, x_{n-1,1} \,<\, x_{n,2} \,<\, x_{n-1,2} \,<\, \cdots \,<\, x_{n,d_n} \,<\, x_{n-1,d_n} \quad\text{when $n$ odd}. \hskip40pt $$

\end{itemize}
\end{thm}

\begin{proof}
From the initial polynomials~\eqref{ini:W}, it is easy to verify Theorem~\ref{thm:RR} for $n\le 2$.
We suppose that $n\ge3$ and proceed inductively.

For every $k\le n-1$, we denote the roots of $W_k(z)$ by $x_{k,1}<x_{k,2}<\cdots<x_{k,d_k}$.  For convenience, we define $x_{k,0}=-\infty$ and $x_{k,d_k+1}=0$, for all $k\le n-1$.  To clarify the interlacing properties, we now consider the signs of the function $W_m(z)$ at $-\infty$ and at the origin, for any $m\ge0$. 
Since $W_m(z)$ is a polynomial of degree $d_m$, with all coefficients non-negative, we deduce that 
\begin{equation}\label{sgn:claw:-infty}
(-1)^{d_m}W_m(-\infty)>0.
\end{equation}
Having the constant term positive implies that 
\begin{equation}\label{sgn:claw:0}
W_m(0)=g_{n,0}>0.
\end{equation}

By the intermediate value theorem and Inequality~(\ref{sgn:claw:-infty}), for  the polynomial~$W_n(z)$ to have~$d_n=\deg{W_n(z)}$ distinct negative roots and for Part~(i) of Theorem~\ref{thm:RR} to hold, it is necessary and sufficient that
\begin{equation}\label{dsr:claw:n:n-2}
(-1)^{d_n+j}W_n(x_{n-2,j})>0 \quad\text{for $1\le j\le d_{n-2}+1$.}
\end{equation}
In fact, for $j=d_{n-2}+1$, 
Inequality~(\ref{dsr:claw:n:n-2}) becomes
\begin{equation} \label{eq:dd1}
(-1)^{d_n+d_{n-2}+1}W_n(0)>0.
\end{equation}
By~(\ref{sgn:claw:0}), Inequality \eqref{eq:dd1} holds if and only if $d_n+d_{n-2}$ is odd,
which is true since
\[
d_n+d_{n-2} \,=\, \llceil\frac{n+1}{2}\rrceil+\llceil\frac{n-1}{2}\rrceil \,=\, 2\llceil\frac{n-1}{2}\rrceil+1.
\]

Now consider any $j$ such that $1\le j\le d_{n-2}$. We are going to prove~(\ref{dsr:claw:n:n-2}).
We will use the particular indicator function $\mathrm{I}_{\mathrm{even}}$, which is defined by
\[
\mathrm{I}_{\mathrm{even}}(n) \,=\, \begin{cases}
1,&\text{if $n$ is even},\\
0,&\text{if $n$ is odd}.
\end{cases}
\] 
Note that $x_{n-2,j}$ is a root of $W_{n-2}(z)$.  By Recurrence~(\ref{rec:W}), we have
\begin{equation}\label{rec:claw:n:n-2}
   W_n(z_{n-2,j}) \,=\, x_{n-2,j}^{\mathrm{I}_{\mathrm{even}}(n)}
   \Bigl(20W_{n-1}(x_{n-2,j})-384x_{n-2,j}W_{n-3}(x_{n-2,j})\Bigr).
\end{equation}
Since $x_{n-2,j}<0$, the factor $x_{n-2,j}^{\mathrm{I}_{\mathrm{even}}(n)}$ contributes $(-1)^{n+1}$ to the sign of the right hand side of~(\ref{rec:claw:n:n-2}). On the other hand,
it is clear that the sign of the parenthesized factor can be determined if both the summands
$20W_{n-1}(x_{n-2,j})$ and $-384x_{n-2,j}W_{n-3}(x_{n-2,j})$ have the same sign.
Therefore, Inequality~(\ref{dsr:claw:n:n-2}) holds if
\begin{align}
(-1)^{d_n+j+n+1}W_{n-1}(x_{n-2,j})&>0,\label{dsr:claw:n-1:n-2}\\
(-1)^{d_n+j+n+1}W_{n-3}(x_{n-2,j})&>0.\label{dsr:claw:n-3:n-2}
\end{align}

By the induction hypothesis on part~(ii) of this theorem, we can substitute $P=W_{n-1}$ and $Q=W_{n-2}$ into Lemma~\ref{lem:interlacing}.  Then Inequality~(\ref{lem:interlacing:P}) gives
\begin{equation}\label{sgn:claw:n-1:n-2}
(-1)^{d_{n-1}+j}W_{n-1}(x_{n-2,j})>0.
\end{equation}
Thus, Inequality~(\ref{dsr:claw:n-1:n-2}) holds if and only if
the total power
\[
d_n+j+n+1+d_{n-1}+j ~=~ \llceil\frac{n+1}{2}\rrceil+\llceil\frac{~n~\vphantom{1}}{2}\rrceil+n+2j+1
\]
of $(-1)$ in~(\ref{dsr:claw:n-1:n-2}) and~(\ref{sgn:claw:n-1:n-2}) is even, which is clear by a simple parity argument. Moreover, again using the induction hypothesis on part~(ii), we can make substitutions $P(x)=W_{n-2}(x)$ and $Q(x)=W_{n-3}(x)$ into Lemma~\ref{lem:interlacing}.  Then  Inequality~(\ref{lem:interlacing:Q}) gives 
\begin{equation}\label{sgn:claw:n-3:n-2}
(-1)^{d_{n-3}+j}W_{n-3}(x_{n-2,j})<0.
\end{equation}
Thus, Inequality~(\ref{dsr:claw:n-3:n-2}) holds if and only if the total power
\begin{equation}\label{totalpower:n-3}
d_n+j+n+1+d_{n-3}+j ~=~ \llceil\frac{n+1}{2}\rrceil+\llceil\frac{n-2}{2}\rrceil+n+2j+1
\end{equation}
of $(-1)$ in~(\ref{dsr:claw:n-3:n-2}) and~(\ref{sgn:claw:n-3:n-2}) is odd, which is also clear by a simple parity argument.  This completes the proof of~(\ref{dsr:claw:n:n-2}), and the proof of Part~(i).

The approach to proving Part (ii) is similar to that used to prove Part (i).  By the intermediate value theorem and Inequality~(\ref{sgn:claw:-infty}), Part (ii) holds if and only if
\begin{equation}\label{dsr:claw:n:n-1}
(-1)^{d_n+j}W_n(x_{n-1,j})>0 \quad\text{for $1\le j\le d_{n-1}$,}
\end{equation}
and also for $j=d_{n-1}+1$ when $n$ is even. In fact, when $n$ is even and $j=d_{n-1}+1$,
we have 
\begin{equation}\label{dsrA:claw:n:n-1}
(-1)^{d_n+d_{n-1}+1}W_n(0)>0.
\end{equation}
By~(\ref{sgn:claw:0}), Inequality \eqref{dsrA:claw:n:n-1} holds if and only if $(-1)^{d_n+d_{n-1}+1}=1$,
which is clear since
\[
d_n+d_{n-1}+1 ~=~ \llceil\frac{n+1}{2}\rrceil+\llceil\frac{~n~\vphantom{1}}{2}\rrceil+1=n+2.
\]

For $1\le j\le d_{n-1}$, we are now going to show~(\ref{dsr:claw:n:n-1}).  By setting $x=x_{n-1,j}$, Recurrence~(\ref{rec:W}) turns into
\begin{equation}\label{rec:claw:n:n-1}
W_n(x_{n-1,j}) ~=~ 8(3-8x_{n-1,j})W_{n-2}(x_{n-1,j})
- 384x_{n-1,j}^{1+\mathrm{I}_{\mathrm{even}}(n)}W_{n-3}(x_{n-1,j}).
\end{equation}
Since $x_{n-1,j}<0$, we see that $8(3-8x_{n-1,j})>0$, and that the factor
$-384x_{n-1,j}^{1+\mathrm{I}_{\mathrm{even}}(n)}$ contributes $(-1)^{n+1}$ to
the sign of the right-hand side of~(\ref{rec:claw:n:n-1}).
Therefore, Inequality~(\ref{dsr:claw:n:n-1}) holds if
\begin{align}
(-1)^{d_n+j}W_{n-2}(x_{n-1,j})&>0,\label{dsr:claw:n-2:n-1}\\
(-1)^{d_n+j+n+1}W_{n-3}(x_{n-1,j})&>0.\label{dsr:claw:n-3:n-1}
\end{align}

Substituting $P(x)=W_{n-1}(x)$ and $Q(x)=W_{n-2}(x)$ into Lemma~\ref{lem:interlacing},
we find that Inequality~(\ref{lem:interlacing:Q}) yields 
\begin{equation}\label{sgn:claw:n-2:n-1}
(-1)^{d_{n-2}\,+\,j}W_{n-2}(x_{n-1,j})<0 \quad\text{when $1\le j\le d_{n-1}$.} 
\end{equation}
Thus, Inequality~(\ref{dsr:claw:n-2:n-1}) holds if and only if the total power
\[
d_n+j+d_{n-2}+j ~=~ \llceil\frac{n+1}{2}\rrceil+\llceil\frac{n-1}{2}\rrceil+2j
\]
of $(-1)$ in~(\ref{dsr:claw:n-2:n-1}) and~(\ref{sgn:claw:n-2:n-1}) is odd, which holds true, obviously, by parity.  On the other hand, by the induction hypothesis on Part~(i) and substituting $P(x)=W_{n-1}(x)$ and $Q(x)=W_{n-3}(x)$ into Lemma~\ref{lem:interlacing}, Inequality~(\ref{lem:interlacing:Q}) becomes 
\begin{equation}\label{sgn:claw:n-3:n-1}
(-1)^{d_{n-3}+j}W_{n-3}(x_{n-1,j})<0.
\end{equation}
Therefore, Inequality~(\ref{dsr:claw:n-3:n-1}) holds if and only if the total power
$d_n+j+n+1+d_{n-3}+j$
of $(-1)$ in~(\ref{dsr:claw:n-3:n-1}) and~(\ref{sgn:claw:n-3:n-1}) is odd,
which coincides with~(\ref{totalpower:n-3}).
This completes the proof of~(\ref{dsr:claw:n:n-1}), ergo the proof of Part~(ii),
and hence the entire theorem.
\end{proof}

\begin{cor}
The sequence of coefficients for every genus polynomial $\Gamma_{Y_n}(z)$ is log-concave. 
\end{cor}

\end{document}